\documentclass[a4paper,12pt]{article}
\usepackage[margin=24mm]{geometry}
\usepackage{mathptmx,amssymb,amsmath,amstext,amsthm,url}

\usepackage[skins, listings]{tcolorbox}
\definecolor{lightgray}{gray}{0.9}
\usepackage{booktabs, parskip}
\definecolor{dnrbl}{rgb}{0,0,0.3}
\definecolor{dnrgr}{rgb}{0,0.3,0}
\definecolor{dnrre}{rgb}{0.5,0,0}

\newtheorem{proposition}{Proposition}

\DeclareMathOperator{\KP}{\mathrm{K}\mskip 0.5mu}
\DeclareMathOperator{\KS}{\mathrm{C}\mskip 0.5mu}
\DeclareMathOperator{\Dim}{\mathrm{Dim}\mskip 0.5mu}
\newcommand{\uhr}{\mskip -3mu \upharpoonright\mskip -3mu}
\newcommand{\cnd}{\mskip 1mu | \mskip 1mu}
\let\le=\leqslant
\let\ge=\geqslant

\begin{document}
\title{The Kučera--Gács Theorem Revisited by Levin}
\date{}
\author{George Barmpalias\thanks{State Key Lab of Computer Science, Institute of Software, Chinese Academy of Sciences, Beijing, China, \texttt{barmpalias@gmail.com} Supported by NSFC grant No. 11971501.}, Alexander Shen\thanks{LIRMM, Univ Montpellier, CNRS, Montpellier, France.  Supported by ANR-21-CE48-0023 FLITTLA grant, \hfil\break \texttt{alexander.shen@lirmm.fr}, \texttt{sasha.shen@gmail.com}, \texttt{https://www.lirmm.fr/\~{}ashen}}}
\maketitle

\begin{abstract}
Leonid Levin~\cite{levin-this} published a new (and very nice) proof of the Kučera--Gács theorem that occupies only a few lines when presented in his style. We try to explain more details and discuss the connection of this proof with image randomness theorems, making explicit some result (see Proposition~\ref{prop:image}) that is implicit in~\cite{levin-this}. Then we review the previous work about the oracle use when reducing a given sequence to another one, and its connection with algorithmic dimension theory.
\end{abstract}

\section{The Kučera--Gács Theorem}

The Kučera--Gács theorem\cite{Kucera1985, Gacs1986} says that \emph{every \textup(infinite\textup) sequence of zeros and ones is Turing-re\-ducible to some Martin-Löf random sequence}. In other words, for every sequence $\alpha$ there exist a Martin-Löf random (with respect to the uniform Bernoulli distribution) sequence $\omega$ and a computable mapping $T$ of the space of binary sequences to itself such that $T\omega=\alpha$.

Let us comment on the notions used in the statement. The notion of Martin-Löf random sequence was defined by Martin-Löf in 1966 \cite{MartinLof1966} and since then has become a standard notion in algorithmic randomness (see, e.g., \cite{LiVitanyi, DowneyHirschfeldt,Nies,suv}).

To define computable mappings, consider an oracle machine that has read access to an infinite bit sequence on the input tape and writes bits sequentially on the output tape. It computes a mapping $T$ with domain $\Omega$ (the space of all infinite binary sequences) and codomain $\Sigma$ (the space of all finite and infinite sequences). Mappings that correspond to oracle machines are called \emph{computable}.

The proofs given by  Kučera and Gács (as well as their subsequent improvements that limit the number of input bits needed to produce $n$ output bits, see below) all follow the same scheme: given $\alpha$, we construct some random sequence $\omega$ and some ad hoc mapping $T$ that maps $\omega$ to $\alpha$ but has no meaning outside this context. However, there is an alternative approach. There is an image randomness theorem (see, e.g., \cite{suv}, or \cite{BienvenuHoyrupShen2017} for more general version) that says that for a computable mapping $T$ that maps a uniform Bernoulli distribution $P$ on $\Omega$ to some \emph{probability distribution} $Q$ \emph{on $\Omega$} (this means that $T(\omega)$ is infinite for $P$-almost all $\omega$), \emph{every sequence that is random with respect to $Q=T(P)$ is a $T$-image of some $P$-random sequence}. The alternative plan is to find some extension of this result to the case where the image measure $T(P)$ is a \emph{semimeasure}, and then derive the Kučera--Gács theorem from this extended version.

Let us explain this plan and the notion of a semimeasure in more details. Consider an arbitrary computable mapping $T\colon \Omega\to\Sigma$. Applying $T$ to a random uniformly distributed point $\omega$ in $\Omega$, we get some random variable $\xi=T(\omega)$ with values in $\Sigma$. For every string $x$ we may consider the probability of the event ``$\xi$ starts with $x$''. We get some non-negative real function $M_T(x)$ defined on all strings. Obviously, $M_T(\Lambda)=1$ for the empty string $\Lambda$, and 
\[
   M_T(x)\ge M_T(x0)+M_T(x1) \eqno(*)
\]
for every string~$x$. The last inequality may not be an equality; the difference
\[
M_T(x)-M_T(x0)-M_T(x1)
\]
is the probability of the event ``$\xi$ is finite and equals $x$'' (since it corresponds to the cases when $x$ is a prefix of $\xi$ but neither $x0$ nor $x1$ are). 

The non-negative functions $M$ on strings that have value $1$ on the empty string and satisfy the inequality $(*)$, are called \emph{semimeasures} and correspond to arbitrary probability distributions on $\Sigma$; we often use the same letter for the probability distribution on $\Sigma$ and the corresponding function on strings. One may ask which semimeasures can be obtained as $M_T$ for a computable mapping $T$, as described above. This question was answered by Levin in 1970~\cite{ZvonkinLevin1970}: the function $M$ should be \emph{lower semicomputable}. This means that the set of pairs $\langle r,x\rangle$ where $x$ is a string and $r$ is a rational number smaller that $M(x)$, is (computably) enumerable.

In the same paper Levin noted that there exists a maximal (up to $O(1)$-factor) function in the class of lower semicomputable semimeasures; it is now called the \emph{continuous a~priori probability}. More precisely, there are many such functions that differ by $O(1)$-factors; we fix one of them and denote it by $\mathsf{M}(x)$. The continuous a~priori probability $\mathsf{M}(x)$ can be used to characterize randomness with respect to computable measures on~$\Omega$. Let $Q$ be a computable measure on $\Omega$, i.e., a lower semicomputable semimeasure such that all finite strings have probability zero:
\[
Q(x)=Q(x0)+Q(x1)
\]
for all~$x$. Levin proved in \cite{Levin1973} that \emph{a sequence $\alpha$ is Martin-Löf random with respect to $Q$ if and only if the ratio $\mathsf{M}(x)/Q(x)$ is bounded for all prefixes $x$ of $\alpha$}. The intuitive meaning: $\alpha$ is \emph{not} random with respect to $Q$ if some prefixes of $\alpha$ are much more likely according to the a priori distribution~$\mathsf{M}$ than according to~$Q$ (the ratio is unbounded). 

Combining this characterization of randomness with the image randomness theorem mentioned above, we come to the following result.

\begin{proposition}
Let $P$ be the uniform Bernoulli distribution on $\Omega$. Let $T\colon\Omega\to\Sigma$ be a computable mapping such that $T(\omega)$ is infinite for $P$-almost $\omega$, and let $Q=T(P)$ be the image distribution. Then the following properties of an infinite sequence $\alpha$ are equivalent:
\begin{itemize}
\item $\mathsf{M}(x)/Q(x)$ is bounded for all prefixes $x$ of $\alpha$;
\item $\alpha=T(\omega)$ for some Martin-Löf random \textup(with respect to $P$\textup) sequence $\omega$. 
\end{itemize}
\end{proposition}
If we could get rid of the restriction that $T(\omega)$ is infinite for almost all $\omega$, thus extending this proposition to the case when $Q$ is a semimeasure, the Kučera--Gács theorem will follow from this generalization. Indeed, the semimeasure $\mathsf{M}$ (being semicomputable) can be represented as $T(P)$ where $T$ is some computable mapping of $\Omega$ to $\Sigma$. Then $\mathsf{M}(x)/Q(x)=1$ is bounded everywhere, and this (hypothetical) generalization would imply that every sequence $\alpha$ is $T(\omega)$ for some random $\omega$.

\section{Levin's approach}

However, this plan does not work: as found in~\cite{BHPS2013}, this generalized statement is not true. Levin~\cite{levin-this} found a way to overcome these difficulties by using a less general statement. Namely, he made the following three observations\footnote{We first provide them in a simplified form that is enough for the Kučera--Gács theorem; see the next section for discussion.} (Propositions~\ref{prop:rounding}--\ref{prop:image}):

\begin{proposition}\label{prop:rounding}
There exists a maximal \textup(up to $O(1)$-factor\textup) lower semicomputable semimeasure whose values are finite binary fractions having at most $2n+O(1)$ bits for strings of length $n$.
\end{proposition}

\begin{proposition}\label{prop:semimeasure}
Every semimeasure $Q$ whose values on $n$-bit strings are finite binary fractions of length at most $2n+O(1)$, is an image of the uniform distribution $P$ by a computable mapping $T\colon \Omega\to\Sigma$ that uses only the first $2n+O(1)$ bits of input to produce $n$ bits of output.
\end{proposition}

The claim means that the $n$ first bits of $T(\omega)$ depend only on $2n+O(1)$ bits of $\omega$. In particular, if $T(\omega)$ has length less than $n$, then the same is true for all $T(\omega')$ if $\omega'$ and $\omega$ have the same first $2n$ bits.

\begin{proposition}\label{prop:image}
Let $T\colon\Omega\to\Sigma$ be a computable mapping that uses only the first $2n+O(1)$ input bits to produce $n$ output bits. Let $P$ be the uniform Bernoulli distribution on~$\Omega$ and let $Q=T(P)$ be the image semimeasure. If the ratio $\mathsf{M}(x)/Q(x)$ is bounded for all prefixes of some infinite sequence $\alpha$, then $\alpha=T(\omega)$ for some Martin-Löf random \textup(with respect to $P$\textup) sequence $\omega$.
\end{proposition}

\begin{proof}[Proof of Proposition~\ref{prop:rounding}.]
We may perform rounding and replace each value on some $n$-bit string by the maximal $(2n+c)$-bit binary fraction that is strictly smaller than this value. Here $c$ is some constant that will be chosen later. This procedure gives a lower semicomputable function (here it is important that we use strict inequalities) with required granularity, but this function may not be a semimeasure. The property
\[
Q(x) \ge Q(x0)+Q(x1)
\]
may be violated after rounding. Indeed, while the right-hand side may only decrease after rounding (and this is not a problem),  the left-hand side also may decrease making the inequality false. The change in the left-hand side is at most $2^{-2n-c}$ (where $n=|x|$), so a small safety margin is enough: if 
\[
Q(x)\ge Q(x0)+Q(x1)+2^{-2n-c},
\]
then after rounding we have a semimeasure. (We should also change the value on the empty string to~$1$, but this change does not violate the other requirements.)

To guarantee this safety margin, we add to the maximal semimeasure $\mathsf{M}$ some other semimeasure with desired safety margin. (Then we have to divide the sum by $2$ to get a semimeasure, but this only increases $c$ by $1$.) For example, the semimeasure $Q(x)=2^{-2|x|}$ has safety margin $2^{-2n}-2\cdot 2^{-2n-2}=2^{-2n-1}$ for $n$-bit strings, so for $c\ge 1$ this is enough.

We need also to check that our semimeasure remains maximal after rounding. Indeed, the safety margin is bigger than the granularity, and the combined effect of the increase and rounding can only increase the value of the semimeasure.
\end{proof}

\begin{proof}[Proof of Proposition~\ref{prop:semimeasure}.]
Here we have to recall the construction of the mapping for a given semimeasure (see \cite{ZvonkinLevin1970} or \cite[Section 5.1]{suv} for details).  This construction is performed in terms of space allocation. Assume that some semimeasure $Q$ is given. For every string $x$ of length~$n$, we allocate to~$x$ some subset of $\Omega$ that is a union of cones (a \emph{cone} $\Omega_y$ is a set of all infinite sequences with a given finite prefix~$y$) and has total measure $Q(x)$. If $Q(x)$ is a multiple of $2^{-2n-c}$, we may use cones $\Omega_y$ of size $2^{-2n-c}$ (with $|y|=2n+c$) when allocating space to $x$. Since $2n+c$ is a monotone function of $n$, for the children of $x$ we may use smaller cones inside the cones allocated to~$x$. In this way we get a mapping where $n$ output bits are determined by $2n+c$ input bits: the preimage of $x$ is the union of cones allocated to $x$, and all these cones have size~$2^{-2n-c}$, so only the first $2n+c$ bits matter.\footnote{A technical clarification: this argument assumes that the approximations to $Q(x)$ are also multiples of $2^{-2n-c}$. This can be achieved by rounding the approximations to the closest binary fraction with $2n+c$ bits; one could also note that the construction provided by Proposition~\ref{prop:rounding} already has this property.} 
\end{proof}

\begin{proof}[Proof of Proposition~\ref{prop:image}.]

This proposition can be proven in different ways; for example, one can adapt the argument from~\cite[Section 5.9.3]{suv}. Levin gives a much simpler argument, but it uses the characterization of Martin-Löf randomness in terms of \emph{expectation-bounded randomness tests} that were introduced by Levin in~\cite{Levin1976} and then studied in~\cite{Gacs1979} (and probably are not as well known as they deserve to be). We recall their definition and properties that we need (see~\cite[Section~3.5]{suv} and \cite{BGHRS2011} for a detailed exposition).

We consider lower semicomputable non-negative real functions on $\Omega$ (the infinite value $+\infty$ is also allowed). The notion of lower semicomputability is an effective version of lower semicontinuity: function $f$ is \emph{lower semicomputable} if the set of pairs $\{\langle r,\omega\rangle: r<f(\omega)\}$ (where $r$ is a rational number and $\omega\in \Omega$) is an effectively open set in $\mathbb{Q}\times\Omega$ (a union of an enumerable family of basic open sets of the form $(p,q)\times \Omega_u$, where $p<q$ are rational numbers and $u$ is a binary string).

The sum of two lower semicomputable functions is lower semicomputable; the same is true for an (effectively given) series of semicomputable functions. A \emph{average-bounded randomness test} is a lower semicomputable function whose integral (in other words, average, or expected value) over $\Omega$ (with uniform measure~$P$) is finite. Taking a mix of all average-bounded randomness tests, we get a \emph{maximal}, or \emph{universal} average-bounded randomness test. It is defined up to $O(1)$-factors; we denote it by $\mathsf{t}$ and may assume that $\mathbf{E}_P\mathsf{t}<1$ (the expectation is less than~$1$). Now Martin-Löf random sequences can be equivalently defined as sequences $\omega\in\Omega$ such that $\mathsf{t}(\omega)<\infty$.

Every non-negative integrable function on $\Omega$ (with finite integral) is a density function of some measure on $\Omega$, so we construct a measure $\tau$ and let $\tau(X)=\int_X \mathsf{t}(\omega)\,dP(\omega)$. The measure $\tau(\Omega)$ of the entire space~$\Omega$ is $\int_\Omega\mathsf{t}<1$. We can transform this measure into a lower semicomputable semimeasure by artificially declaring $\mathsf{t}(\Lambda)=1$ for the corresponding function on strings.\footnote{This semimeasure has a natural interpretation. Namely,  the exact expression (see~\cite{Gacs1979,BGHRS2011}) for $\mathsf{t}$ says that it corresponds (up to a $O(1)$-factor) to the following random process: we first generate a finite string according to discrete a priori probability and then start adding random bits. But this is not important for the argument.} 

Now we are ready to prove Proposition~\ref{prop:image}. In addition to $Q=T(P)$ we consider the image measure $M'=T(\tau)$. It is easy to check that this measure also can be extended to a lower semicomputable semimeasure, and therefore it is does not exceed $c\mathsf{M}$ for some constant~$c$.

Recall our assumption: the ratio $\mathsf{M}(x)/Q(x)$ is bounded for all prefixes $x$ of~$\alpha$. Let $c'$ be this bound, so $M(x)\le c'Q(x)$ for all prefixes $x$ of~$\alpha$. Combining this with the previous inequality, we conclude that $M'=T(\tau) \le c\mathsf{M}\le cc'Q$. Now let us look more closely at the values of $M'(x)$ and $Q(x)$ for some $x$ that is a prefix of~$\alpha$. Since $M'=T(\tau)$ and $Q=T(P)$, both values are (different) measures of the set $U=T^{-1}(\Omega_x)$, the set of all sequences that are mapped by $T$ into some extensions of $x$. The first is $\tau(U)=\int_U \mathsf{t}(\omega)\,dP(\omega)$, the second is just $P(U)$.  The inequality $M'<cc'Q$ implies that there exists a sequence $\omega\in U$ (in other words, $T(\omega)$ starts with $x$) such that $\mathsf{t}(\omega)\le cc'$. We denote this sequence by $\omega_x$ indicating that this sequence depends on $x$ (while constants $c$ and $c'$ do not depend on~$x$).

Taking longer and longer prefixes $x$ of $\alpha$, we get an infinite sequence of corresponding $\omega_x$. We know that $\mathsf{t}(\omega_x)\le cc'$ for all $\omega_x$, and that $T(\omega_x)$ starts with $x$. Let $x_i$ be the prefix $\alpha$ that has length $i$, and let $\omega_i=\omega_{x_i}$ be the corresponding sequence. Using compactness, we take a limit point $\omega$ of the sequence $\omega_i$. Let us show that $\omega$ is a sequence that we need: a random sequence such that $T(\omega)=\alpha$. Indeed, $\omega$ is random since the set of all sequences~$\xi$ such that $\mathsf{t}(\xi)\le cc'$ is closed (since $\mathsf{t}$ is lower semicomputable and therefore lower semicontinuous). And for every prefix $x$ the corresponding set $U=T^{-1}(\Omega_x)$ is not only open (as it will be for any computable mapping) but also closed due to our assumption ($n$ output bits are determined by $2n+O(1)$ input bits), and contains all $\omega_i$ starting from $i=|x|$, and therefore all limit points of the sequence $\omega_0,\omega_1$,\ldots. Proposition~\ref{prop:image} is proven.
\end{proof}

\section{Remarks about Levin's argument}

\subsubsection*{Stronger versions} 

We formulated Propositions~\ref{prop:rounding}--\ref{prop:image} in the simplest form that is enough for the Kučera--Gács theorem. Levin's formulations are stronger. 

\emph{Proposition}~\ref{prop:rounding}.
Here we may easily replace $2n$ by $n+d(n)$ where $d(n)$ is a computable  function such that $\sum 2^{-d(n)}\le 1$. For that we use the safety margin $2^{-n-d(n)}$ for strings of length $n$, and the total additional weight is $\sum 2^{-d(n)}$, since we have $2^n$ strings of length $n$ with weight $2^{-n-d(n)}$ each. Moreover, Levin notes that we may let the granularity depend not only on the length of $x$, but on $x$ itself, using $\KP(x)+O(1)$ bound:
\begin{quote}
\textit{%
   There exists a constant $c$ and a maximal \textup(up to $O(1)$ factor\textup) lower semicomputable semimeasure whose value at every string $x$ is a finite binary fraction having at most $\KP(x)+c$ bits.
}
\end{quote} 
To prove this, we use $\mathsf{m}(x)=2^{-\KP(x)}$ (where $\KP(x)$ is the prefix complexity of $x$) to provide a safety margin for string~$x$. Consider a semimeasure $S(x)=\sum_{x\preceq y} \mathsf{m}(y)$, where the sum is taken over all strings $y$ that have $x$ as a prefix. As usual, we let $S(\Lambda)=1$ to get a semimeasure. The semimeasure~$S$ corresponds to the following process: generate a finite string according to $\mathsf{m}$, and stop (so $S$ is concentrated on finite sequences).

Now consider the average $M'(x)=[\mathsf{M}(x)+S(x)]/2$, where $\mathsf{M}$ is a maximal lower semicomputable semimeasure. It is maximal up to an $O(1)$ factor (since $\mathsf{M}$ is maximal). The safety margin $M'(x)-M'(x0)-M'(x1)$ for rounding is at least $\textsf{m}(x)/2$ (thanks to $S$). Let us then round $M'(x)$ up to $\KP(x)+c$ bits, for some large enough constant $c$: let $\widehat{\mathsf{M}}(x)$ be the maximal integer multiple of $2^{-\KP(x)-c}$ that is strictly smaller than $M'(x)$ (and $\widehat{\mathsf{M}}(\Lambda)=1$, as usual).

For large enough $c$ the safety margin $\mathsf{m}(x)/2$ exceeds the rounding errors, so $\widehat{\mathsf{M}}$ is a semimeasure. However, some care is needed to show that $\widehat{\mathsf{M}}$ is lower semicomputable. The problem is that we cannot compute the value of $\KP(x)$, we have only the approximations from above for $\KP(x)$ that can decrease at any time. To construct the approximations from below to $ \widehat{\mathsf{M}}$, we consider the current approximation $k$ for $\KP(x)$ and take the maximal multiple of $2^{-k-c}$ that is strictly less than the current approximation from below to $M'(x)$. At some point $k$ reaches the final value, i.e., $\KP(x)$, so the sequence converges to the maximal multiple of $2^{\KP(x)-c}$ that is strictly smaller than $M'(x)$. However, we need the sequence of approximations to be increasing, and now this is not obvious, since $k$ can decrease, and bigger granularity may lead to smaller approximations. To achieve the monotonicity, we synchronize the increases in the approximations for $\mathsf{m}(x)$ and the corresponding increases in the approximations for $M'(x)$ with decreases in $k$. If $c$ is large enough, the increase in $M'$ compensates the effects of increased granularity, and this finishes the argument.

For \emph{Proposition}~\ref{prop:semimeasure} we may replace $2n+O(1)$ by arbitrary total computable non-de\-creas\-ing function. The computability is needed to make the construction effective, and the function should be non-decreasing since we need to allocate intervals for $x0$ and $x1$ inside the intervals already allocated for $x$. (It seems that this requirement is not mentioned in~\cite[last paragraph of Section~7.1]{levin-this}, but it is not clear how we can proceed without it.) We may also use a total computable function $t(x)$, depending on the \emph{string~$x$}, not only on its \emph{length}, if  $t(x)\le t(x0)$ and $t(x)\le t(x1)$ for all strings~$x$. Moreover, we may consider also partial computable functions $t$.

\begin{quote}
\emph{%
Let $t$ be a computable partial function on strings, with natural values. Assume that the domain of $t$ is a subtree \textup(if $t(x)$ is defined, then $t(y)$ is defined for all prefixes $y\preceq x$\textup), and inequalities $t(x)\le t(x0)$ and $t(x)\le t(x1)$ hold when both sides are defined. Let $Q$ be a lower semicomputable semimeasure. Assume that $Q(x)$ is an integer multiple of $2^{-t(x)}$ for all strings $x$ in the domain of $t$. Then there exist a computable mapping $T\colon \Omega\to\Sigma$ such that for every $x$ where $t(x)$ is defined, the following properties hold:
\begin{itemize}
\item the probability of the event ``$T(\omega)$ starts with $x$'' equals $Q(x)$;
\item the set of all $\omega$ such that $T(\omega)$ starts with $x$, is a union of intervals of size $2^{-t(x)}$.
\end{itemize} 
}%
\end{quote}
 Proving this stronger version of Proposition~\ref{prop:semimeasure}, we may again assume that the growing approximations for $Q(x)$ are multiples of $2^{-t(x)}$ for strings~$x$ where $t(x)$ is defined (the other values of $Q$ are not used). Then we inductively construct the $T$-preimage of $\Omega_x$ as a union of intervals of size $2^{-t(x)}$, choosing those intervals inside the intervals for the parent of $x$. Note that $t(x0)$ and $t(x1)$ may be different, and we may have to allocate space for $x0$ (intervals of size $2^{-t(x0)}$) knowing only the values of $t(x)$ and $t(x0)$ but not $t(x1)$, or vice versa. For that we start filling an interval of size $2^{-t(x)}$ by intervals of size $2^{-t(x0)}$ from left to right, and by intervals of size $2^{-t(x1)}$ from right to left; when there is not enough space for the next interval in some direction, we continue in the same direction in the other interval of size $2^{-t(x)}$ (and use the remaining space in the first interval, if there is some, for the other direction).

 \emph{Proposition}~\ref{prop:image} can be generalized in the following way (and the proof remains essentially unchanged):
\begin{quote}
\emph{%
Let $t$ be a computable partial function on strings whose domain is a subtree. Let $T$ be a computable mapping $\Omega\to\Sigma$, and for every $x$ in the domain of $t$ the event ``$T(\omega)$ starts with $x$'' is a union of intervals of size $2^{-t(x)}$. Let $P$ be the uniform distribution on $\Omega$ and let $Q=T(P)$ be the image semimeasure. If all prefixes of some infinite sequence $\alpha$ are in the domain of $t$ and the ratio $\mathsf{M(x)}/Q(x)$ is bounded for them, then $\alpha=T(\omega)$ for some Martin-L\"of random \textup(with respect to $P$\textup) sequence~$\omega$.
}%
\end{quote} 

Combining these observations, we get an improved version of the Kučera--Gács theorem:
\begin{quote}
\emph{%
 If $t(x)$ is a partial computable function that is an upper bound for $\KP(x)$ \textup(where defined\textup) such that the domain of $t$ is a subtree and $t(x)\le t(x0)$ and $t(x)\le t(x1)$ if both sides are defined, then for every infinite sequence~$\alpha$ whose prefixes belong to the domain of $t$ there exists a Martin-Löf random sequence $\omega$ and a machine $M$ that computes $\alpha$ given oracle $\omega$ and uses at most $t(x)+O(1)$ bits of $\omega$ when producing a prefix $x$ of $\alpha$. 
}%
\end{quote}
The last statement means that machine $M$ produces the same prefix $x$ when applied to every infinite sequence that has the same $t(x)$ first bits as the sequence~$\omega$.

Levin states this results for total $t$, but we give its natural generalization to compare
Levin's argument with some other line of research discussed in the next sections: the monotonic compression.

\subsubsection*{Partially continuous transformations} 

This notion (used by Levin in~\cite{levin-this}) is a generalization of the notion of an oracle machine, and a notion of a  (monotone) computable operation defined by him in~\cite{Levin1973}. Informally speaking, an oracle machine uses the bits from the oracle tape to produce output bits: the more it knows about the input, the more output bits are produced. If we want to consider finite and infinite inputs, we arrive to the notion of a (monotone) computable operation. Now Levin makes the next generalization step: a partially continuous transformation provides more information about the output if given more information about the input. Such a transformation is an effectively closed set in $\Omega\times\Omega$ and can be defined by a computable sequence of statements of the following form: ``if the input sequence belongs to this cone, then the output sequence cannot belong to that cone''. In this way we can do more than just specify output bits sequentially: transformation may say, for example, that $7$th bit is zero or that $5$th and $16$th bits are equal without claiming anything about other bits. 

The notion of partially continuous transformation is natural in itself, but is not strictly necessary for the argument: we may consider only oracle machines (as usual) and semimeasures defined on finite strings (and not on arbitrary clopen subsets of $\Omega$).

\section{Oracle use and compression}

One may ask what could be the finite version of  the Kučera--Gács theorem. The natural candidate says that every (finite) string $x$ can be computed by a simple program $f$ from a random string~$y$. This is true for obvious reasons: let $y$ be the shortest description of $x$ (so the length of~$y$ equals the Kolmogorov complexity of $x$), and let the simple program $f$ be the decompressor used in the definition of Kolmogorov complexity. One can check easily that the shortest description is incompressible, so we get the desired result.

Can this approach be used for the Kučera--Gács theorem? Here we have some infinite sequence $\alpha$. We may find the shortest descriptions for all the prefixes of $\alpha$, but these shortest descriptions may be unrelated to each other and are not necessarily prefixes of one sequence. We need some other type of compression where the compressed version is a sequence $\omega$ and its prefixes are used by the decompressor to produce prefixes of $\alpha$. The natural definition of this type was suggested by Kobayashi~\cite{Kobayashi1981}. Consider a sequence $\alpha$ that is computed by an oracle machine that has access to oracle $\omega$. For every $n$, we consider the prefix of $\omega$ that is used while computing the first $n$ bits of $\alpha$. (We may assume that $\alpha$ is printed bit by bit on the  write-only one-directional output tape while $\omega$ is provided on a read-only one-directional input tape, so you cannot read some bit of $\omega$ without reading all the previous bits.) The length of the prefix of the input tape (as a function of the length of the prefix of the output tape) is called the \emph{oracle use}\footnote{Technically speaking, this notion of oracle use is a bit more restrictive than before, when we required that $n$ output bits depend only on some number of oracle bits:  it may happen that machine reads more bits than it is actually necessary to determine the output bits. For example, the machine can copy its input to its output and read $2n$ input bits before producing $n$ output bits. Moreover, one can construct a machine whose oracle use in the previous sense is smaller than the oracle use (in the new sense) for \emph{every machine computing the same mapping $\Omega\to\Sigma$}. Indeed, let $x_0,x_1,x_2,\ldots$ be a computable prefix-free sequence (say, $x_n=0^n1$), and let $P$ and $Q$ be two inseparable enumerable sets of natural numbers. Consider the mapping $T\colon\Omega\to\Sigma$ that maps all sequences with prefix $x_i0$ to $x_i0$ for every $i$; if $i\in P$, all sequences with prefix $x_i1$ are also mapped to $x_i0$; if $i\in Q$, all sequences with prefix $x_i1$ are mapped to $x_i1$.  If $i\in P$, the prefix $x_i$ is enough to guarantee the output $x_i0$; if $i\in Q$, one needs to read the bit after $x_i$ to produce the correct output. Assume that some machine $M$ computes $T$; then, if it does not read the bit after $x_i$ for $i\in P$,  we can separate $P$ and $Q$ by looking whether $M$ reads more bits after $x_i$ or not. \par  However, Levin's arguments (and the results mentioned below) can be adapted to this stronger notion in the same way as in~\cite[Theorem 50, p.86]{suv}.} when computing $\alpha$ with oracle~$\omega$. For example, for an oracle machine that copies input $\omega$ to output $\alpha=\omega$ bit by bit, the oracle use is the identity function $n\mapsto n$. For a machine that remembers $c$ first bits of $\alpha$ in its program and copies the rest from $\omega$, the oracle use is $n\mapsto \max(0,n-c)$.  A similar type of compression was considered by Ryabko in~\cite{Ryabko1984} and \cite{Ryabko1986}. 

It would be nice to prove the Kučera--Gács theorem in the following way: for a given (arbitrary)~$\alpha$ we find some $\omega$ that decompresses into $\alpha$ with minimal oracle use, and then derive from this minimality that $\omega$ is random. However, no argument of this type is currently known. What can be done is different: (1)~knowing the upper bounds for the complexity of the prefixes of~$\alpha$, we can construct a sequence $\omega$ that computes $\alpha$ with bounded oracle use, and then (2)~reduce $\omega$ to some random $\omega'$ not increasing the oracle use too much. The composition of these two reductions is an oracle machine that computes $\alpha$ from a random $\omega'$ with bounded oracle use (thus coming to the same result in a different way). 

Let us describe what was done in these two directions, starting with the second one. (Then we will discuss arguments that directly construct a reduction to a random sequence, thus avoiding this two-step procedure.)

\subsubsection*{Original proofs of Gács and Kučera and their optimal implementations}

Gács (1986) in \cite{Gacs1986} and  Kučera (1985) in \cite{Kucera1985} showed that every binary sequence can be computed using a random sequence as an oracle, with computable oracle use.
The method in \cite{Gacs1986} is more sophisticated, and gives the bound  $n+O(\sqrt{n} \log n)$ which is explicitly stated, even in the abstract.\footnote{Gács \cite{Gacs1986} gives $n+3\sqrt{n} \log n$, which can be refined to $n+\sqrt{n} \log n$, as shown in \cite{Barmpalias-Lewis-Pye2020}.}
What is important here, as soon will be clear, is that the overhead $\sqrt{n} \log n$ is $o(n)$ (small compared to $n$). The argument in \cite{Gacs1986} can be described as \emph{block-coding}, where the source~$\alpha$ is broken into blocks (the $i$th block has length $i$) and each block is coded by the corresponding block of a random~$\omega$. In contrast to that,
\cite{Kucera1985} did not break $\alpha$ into blocks, and \emph{coded each bit  separately} (by a block of a random sequence). As a consequence, \cite{Kucera1985}  gives significantly larger oracle use. Kučera was not interested in precise or optimal oracle-bounds and, although $n^2$ is often quoted as the oracle use in this method, Barmpalias and Lewis-Pye~\cite{Barmpalias-Lewis-Pye2020} demonstrated that $n\log n$ is the actual oracle use in  an optimal implementation of Kučera's construction. It was also demonstrated in \cite{Barmpalias-Lewis-Pye2020}, using the analysis of Merkle and Mihailovi\'{c} (2004) from~\cite{Merkle-Mihailovic2004}, that the bound $n+\sqrt{n} \log n$ in \cite{Gacs1986}  cannot be improved using the same method,  by changing the block lengths or other parameters of the construction.

Note that the oracle use in the coding discussed above is \emph{oblivious}: it does not depend on the source being coded (a sequence that we reduce to a random one). Note also that the we have $n+o(n)$ oracle use; this will be important for the next results. (This bound can be improved, see below; Levin's argument discussed above can also be used to get $n+2\log n+O(1)$ bound or even better ones.)

\subsubsection*{How much can we compress a given sequence?}

Now we concentrate on the other side of the question: given a sequence $\alpha$, we want to find some other sequence $\omega$ (random or not) and an oracle machine that computes $\alpha$ with oracle $\omega$, with minimal oracle use. What can we achieve, and what are the obstacles for the high compression?

If the $n$-bit prefix of $\alpha$ (denoted by $\alpha \uhr n$) is computed by a machine with oracle $\omega$, and the oracle use is $u_n$, then the Kolmogorov complexity of this prefix exceeds the complexity of $\omega\uhr u_n$ at most by $O(\log n)$, since $\alpha\uhr n$ can be computed from $n$ and $\omega\uhr u_n$:
\[
\KP(\alpha \uhr n)\le \KP(\omega\uhr u_n)+O(\log n) \le u_n+O(\log n+\log u_n).
\]
Therefore, the oracle use has asymptotic lower bounds:
\[
\liminf \frac{u_n}{n}\  \ge\  \liminf \frac{\KP(\alpha\uhr n)}{n}, \qquad
\limsup \frac{u_n}{n}\  \ge \ \limsup \frac{\KP(\alpha\uhr n)}{n}.
\]
The quantities in the right hand sides of these inequalities are known in algorithmic dimension theory (see \cite{Lutz2000,Mayordomo2002,Lutz2003,Athreya2007}) as \emph{effective Hausdorff dimension} $\dim(\alpha)$ (the first, one, with $\liminf$) and \emph{effective packing dimension}  $\Dim(\alpha)$ (the second one, with $\limsup$) of the sequence $\alpha$. Here we used prefix complexity in the definitions, but one can also use plain complexity $\KS(\alpha\uhr n)$, since plain and prefix complexities differ at most by $O(\log n)$ for $n$-bit strings.

These bounds are tight. For the first bound it was essentially shown by Ryabko in~\cite{Ryabko1984,Ryabko1986}; his reasoning goes as follows. Assume that $\liminf \frac{\KP(\alpha\upharpoonright n)}{n}<d$ for some $d$. Then the sequence $\alpha$ has prefixes of lengths $n_1<n_2<n_3\ldots$ such that $\KP(\alpha\uhr n_i)< dn_i$ for all~$i$. We may assume without loss of generality that $n_i$ grow fast, so $n_1+n_2+\ldots+n_i=o(n_{i+1})$. Then we can concatenate the optimal prefix free descriptions of $\alpha\uhr n_1$, $\alpha\uhr n_2$, \dots\ into an infinite sequence $\omega$. Using $\omega$ as oracle, we compute all $u\uhr n_i$ using at most $d(n_1+\ldots+n_i)=dn_i+o(n_i)$ bits. This implies that $\liminf \frac{u_n}{n}<d$ for the oracle usage $u_n$ in this oracle computation.

It is important that we are use Hausdorff dimension and $\liminf$; this argument does not give any bound for packing dimensions since we have no way to compute prefixes of $\alpha$ that have intermediate lengths (between $n_i$ and $n_{i+1}$) with bounded oracle use. The bound for packing dimension requires a bit more complicated argument that was suggested by Doty~\cite{Doty2006,Doty2007,Doty2008}. This argument goes as follows.

Instead of using fast increasing lengths, let us split $\alpha$ into blocks $\alpha=a_1a_2a_3\ldots$ in such a way that $a_i$ has length $i$. These blocks are ``not too small and not too large'' in the following sense: (1)~every prefix of $\alpha$ of length $n$ can be extended to some bigger prefix $a_1a_2\ldots a_k$ of length at most $n+O(\sqrt{n})$ (``blocks are not too large''), and, at the same time, $k=O(\sqrt{n})$ (``blocks are not too small''). Using the Kolmogorov--Levin formula for complexity of pairs, we note that
\[
\KP(a_1a_2\ldots a_k) = \KP(a_1)+\KP(a_2\cnd a_1)+\ldots \KP(a_k\cnd a_1\ldots a_{k-1}) + O(k\log |a_1\ldots a_k|).
\]
Here $\KP(u\cnd v)$ stands for (prefix) conditional complexity of $u$ given $v$; the logarithmic errors in $k$ applications of the Kolmogorov--Levin formula are accumulated in the $O(k\log |a_1\ldots a_k|)$ term.

Consider now the shortest prefix-free description $w_i$ for $a_i$ given $a_1\ldots a_{i-1}$, and a sequence $\omega=w_1w_2w_3\ldots$. It can be decoded from left to right, and we need 
\[
|w_1|+\ldots+|w_k|=
\KP(a_1)+\ldots+\KP(a_k\cnd a_1\ldots a_{k-1})=\KP(a_1\ldots a_k)+o(|a_1\ldots a_k|)
\]
bits to decode $a_1\ldots a_k$. This inequality works only on the block boundaries; for arbitrary prefix of $\alpha$ we need to restore the partially used block completely, but (recall that blocks are not too large) this is again $o(n)$ overhead. 

Combining these observations with the Kučera--Gács theorem (with $n+o(n)$ bound for the oracle use), we get the following result from~\cite{Doty2006}: for every sequence $\alpha$ there exists a Martin-L\"of random sequence $\omega$ and an oracle machine $M$ that transforms $\omega$ to $\alpha$ with oracle use $u_n$ such that
\[
\liminf \frac{u_n}{n}\ = \ \liminf \frac{\KP(\alpha\uhr n)}{n}, \qquad
\limsup \frac{u_n}{n}\  = \ \limsup \frac{\KP(\alpha\uhr n)}{n}.
\]
Note that the machine $M$ can be fixed in advance (before $\alpha$ is given): consider the universal oracle machine that first reads some self-delimited description of the other machine and then simulates this machine on the rest of the oracle tape.

\begin{table}
\begin{center}
\begingroup\setlength{\fboxsep}{7pt}
\colorbox{lightgray}{%
  \begin{tabular}{ll}
\toprule
\textbf{Coding method} & \hspace{0.3cm} \textbf{Oracle use }  \\[0.5ex]
\midrule
Kučera (1985) \cite{Kucera1985}&\hspace{0.3cm} $n\log n$\\[1.5ex]
Gács (1986) \cite{Gacs1986}&\hspace{0.3cm} $n+\sqrt{n} \log n$    \\[2ex]
Ryabko (1986)~\cite{Ryabko1986} and Doty (2008)~\cite{Doty2008}&\hspace{0.3cm} in $\big(n\cdot \dim (\alpha), n\cdot \Dim(\alpha)\big)$  \\[2ex]
Barmpalias and Lewis-Pye (2016)~\cite{Barmpalias-Lewis-Pye2018}&\hspace{0.3cm} $n+g(n)$ \ for any fast-growing $g\le_T\mathbf{0}$ \\[2ex]
Barmpalias, Lewis-Pye, Fang (2016)~\cite{Barmpalias-Lewis-Fang2016}& \hspace{0.3cm}any fast-growing $g\le_T\mathbf{0}$, if $\alpha$ is c.e. \\[2ex]
Barmpalias and Downey (2017)~\cite{Barmpalias-Downey2017}& \hspace{0.3cm}\textrm{$\min_{i\ge n} \KP(\alpha\uhr i)$\  \ if $\alpha$ is left-c.e.}   \\[2ex]
Barmpalias and Lewis-Pye (2019) \cite[I.4]{Barmpalias-Lewis-Pye2019}&\hspace{0.3cm} any computable upper bound of $\KP(\alpha\uhr n)$  \\[2ex]
Barmpalias and Lewis-Pye (2019) \cite[I.5]{Barmpalias-Lewis-Pye2019}& \hspace{0.3cm}$\min_{i\ge n} [\KP(\alpha\uhr i)+\log i]$   \\[2ex]
 \bottomrule
 \end{tabular}
}\endgroup
\caption{Oracle use for different coding methods. (Here $g\le_T\mathbf{0}$ means that $g$ is computable; $\alpha$ is considered as a set of integers when we assume that $\alpha$ is computably enumerable (c.e.), and as a real when we assume that $\alpha$ is left computably enumerable.) In all the constructions, except for Ryabko's and Doty's ones, a reduction to random oracle is directly provided; in these two constructions an additional layer is needed.}
\label{propofas}
\end{center}
\end{table}

\subsubsection*{Lower bounds on the oracle use}
After analyzing the optimal oracle use via the original methods of Kučera and Gács , one may wonder about lower bounds: the amount of the oracle use that is necessary in computations from random sequences. Downey and Hirschfeldt (2010)  showed  \cite[Theorem 9.13.2]{DowneyHirschfeldt} that constant overhead is not sufficient: there exist sequences that are not computable from any random sequence with oracle use $n+\textrm{O}(1)$. Barmpalias, Lewis-Pye, Teutsch (2016) obtained better lower bounds in~\cite{Barmpalias-Lewis-Teutsch2016}  (and these bounds turned out to be tight, see below).

To formulate their results, let us introduce the following definition: A non-decreasing $g$ is \emph{fast-growing} if  $\sum_{i} 2^{-g(i)}<\infty$, and \emph{slow-growing} otherwise.
Typical examples of slow-growing and fast-growing functions are $n\mapsto\log n$ and $n\mapsto 2\log n$, respectively.

It was shown in \cite{Barmpalias-Lewis-Teutsch2016} that for each slow-growing computable $g$, there exists a sequence $\alpha$ that is not computable (from any random oracle) with oracle use $n+g(n)$.

This result seemed quite weak at that time; for example, it does not prevent reductions to a random sequence with overhead as low as $2\log n$, in sharp contrast  to $\sqrt{n} \log n$ which was known to be the best possible upper bound for the methods used in \cite{Gacs1986,Kucera1985}. Surprisingly, this lower bound turned out to be tight.

\subsubsection*{Optimal oblivious oracle use} 

Barmpalias and Lewis-Pye (2016)~\cite{Barmpalias-Lewis-Pye2018} matched the lower-bounds of \cite{Barmpalias-Lewis-Teutsch2016}  by showing that, for each computable fast-growing $g$, every sequence $\alpha$ is reducible to some random one with oracle use $n+g(n)$. So in the oblivious case, a complete characterization of the possible oracle uses in computations from randoms was given in \cite{Barmpalias-Lewis-Pye2018}. Note that the same result can now be obtained in a different way using Levin's argument, as we have discussed.

It is interesting to note that the same characterization applies to the case of computation of left-c.e.\ reals by random left c.e.\ reals: this was already established in \cite{Barmpalias-Lewis-Fang2016} and motivated reaching the optimal overheads of \cite{Barmpalias-Lewis-Pye2018}, although they were significantly smaller than known ones at the time.\footnote{We will come back to this while discussing  the optimal adaptive coding into random sequences.}

As an algorithm, the coding in \cite{Barmpalias-Lewis-Pye2018} is greedy and, in a sense, straightforward. The correctness proof, however, is more sophisticated  than in the previous arguments due to its global character. Without getting into specifics, we can demonstrate the difference between \cite{Barmpalias-Lewis-Pye2018} and  the original argument from~\cite{Gacs1986,Kucera1985} by exposing a key injectivity property of the code-map that exists in the older methods (making their verification easier) and is missing in \cite{Barmpalias-Lewis-Pye2018}.

\subsubsection*{Differences in oblivious coding methods} 

The proofs of the Kučera--Gács theorem (both in~\cite{Gacs1986,Kucera1985} and~\cite{Barmpalias-Lewis-Pye2018}) construct some oracle machine $\Phi$ such that for every $\alpha$ there exists some random $\omega$ such that $\Phi(\omega)=\alpha$. Moreover, this random $\omega$ has bounded randomness deficiency: we fix an effectively open set $V$ of small measure that contains all non-random sequences (e.g., some level of the Martin-L\"of's randomness test) and ensure that for every $\alpha$ there is some $\omega\notin V$ such that $\Phi(\omega)=\alpha$.

All these proofs (in \cite{Gacs1986,Kucera1985} as well as in \cite{Barmpalias-Lewis-Pye2018}) have the following \emph{injectivity property}: for each $\alpha$ there exists a unique $\omega\notin V$ such that $\Phi(\omega)=\alpha$. If $u_n$ is the oracle use in this proof, the machine $\Phi$ uses only the first $u_n$ bits of $\omega$ when computing the first $n$ bits of $\alpha$. So we get some mapping of $u_n$-bit strings to $n$-bit strings. For the construction from \cite{Kucera1985} this mapping on finite strings is also injective (\emph{local injectivity}): if $\alpha'$ has the same first $n$ bit as $\alpha$, and $\omega$ and $\omega'$ are oracle sequences that let $\Phi$ compute $\alpha$ and $\alpha$, then $\omega$ and $\omega'$ share the same first $u_n$ bits. The construction in~\cite{Kucera1985} can be vaguely described as follows: we embed sparsely the full binary tree (where $\alpha$ is one of the branches) into a full binary tree, mapping nodes of height $n$ to nodes of height $u_n$ from which they could be decoded. When the effectively open set $V$ shadows some part of the embedded tree, we restore this part inside the space that is still available; this is possible since the embedding is sparse. The sequence $\omega$ is then the image of $\alpha$ under the limit embedding.

The construction in~\cite{Gacs1986} improves the oracle use by grouping bits of $\alpha$ in blocks ($n$th block has length $n$). This replaces the binary tree for $\alpha$ by a tree when nodes of height $n$ have $2^n$ children; this tree is then sparsely embedded in the full binary tree. This construction uses space more efficiently; the local injectivity holds for $n$ that are the block boundaries. 

The construction in \cite{Barmpalias-Lewis-Pye2018} is quite different and very far from being locally injective: a single source-string $\alpha\uhr n$ may have unpredictably many random code-strings, all of which are necessary for making the induced map on the infinite sequences surjective: different $\omega\uhr u_n$ may be needed for different $\alpha$ with the same $\alpha\uhr n$. One could conjecture that this is necessary to get oracle use $u_n=n+O(\log n)$. In the words of Kučera\footnote{Said during the workshop {\em Algorithmic Randomness Interacts with Analysis and Ergodic Theory} at the Banff International Research Station for Mathematical Innovation and Discovery (BIRS), of Casa Matemática Oaxaca, December 2016.}, this type of coding  is {\em global} in contrast to  \cite{Gacs1986}, \cite{Kucera1985} which is {\em local}.

Let us note that Levin's construction is not even globally injective.

\subsubsection*{Toward the optimal adaptive oracle use.} 
A natural conjecture is that \emph{for each $\alpha$ there exists a random $\omega$ which computes $\alpha$ with oracle use at most $K(\alpha \uhr n)$.} Note that $K(\alpha\uhr n)$ does not have the form of an oracle use function as it is not monotone, but we could consider instead the monotone function $\min_{i\ge n} K(\alpha\uhr i)$. (We may also try to compare the oracle use with monotone or a priori complexity; both are monotone.)

As we have seen, this question also makes sense without the requirement for $\omega$ to be random.  (Then we add a second layer that uses the oblivious encoding considered above, as we have discussed.) Note that this bound is much stronger than the results of \cite{Ryabko1986,Doty2008}, where all the bounds are up to $o(n)$.

So why one would make such a strong guess? Here we bring up the analogy with the case of computing left-c.e.\ reals from random one, as we did in the case of the oblivious oracle use:\footnote{This is only a heuristic argument; there is no deeper reason why the bounds for the left-c.e.\ reals would match the general case.} Barmpalias and Downey (2017)~\cite{Barmpalias-Downey2017} showed that every left-c.e.\ real $\alpha$ is computable by Chaitin's $\Omega$ (and any other random left-c.e.\ real) with oracle use $\min_{i\ge n} K(\alpha\uhr i)$.

This conjecture remains unproven; two partial results in this direction were obtained by Barmpalias and Lewis-Pye (2019) in \cite{Barmpalias-Lewis-Pye2019}. The first says that if $K^{\ast}$ is a partial computable function that is an upper bound for $\KP$ (at the points where $K^{\ast}$ is defined), then \emph{every sequence $\alpha$ such that $K^{\ast}$ is defined on all prefixes of $\alpha$, is computable from some random sequence with oracle use $\min_{i\ge n} K^{\ast}(\alpha \uhr i)$.} The argument used in~\cite{Barmpalias-Lewis-Pye2019} is based on some version of Kraft--Chaitin lemma that extends it to hierarchical space allocation requests. It directly constructs a random sequence $\omega$ that computes $\alpha$ with bounded oracle use (avoiding the two-step procedure described earlier and the corresponding overhead).

Note that $K^{\ast}$ has to be computable, while $\KP$ is only upper semicomputable (and also non-monotonic). This presented certain issues in the argument of \cite{Barmpalias-Lewis-Pye2019}, if we want to get an upper bound for oracle use in terms of $\KP$. They were only resolved by adding a logarithmic overhead: it is shown in~\cite{Barmpalias-Lewis-Pye2019} that \emph{every sequence $\alpha$ is computable from some random sequence with oracle use  $\min_{i\ge n} [\KP(x\uhr i) +\log i]$.} 

Both results come quite close to the conjecture formulated above, but the question whether the conjecture is true, remains open.

\begin{table}[h]\begin{center}
\begingroup\setlength{\fboxsep}{7pt}
\colorbox{lightgray}{%
  \begin{tabular}{lcc}
\toprule
\textbf{Coding} & \textbf{Type}& \hspace{0.3cm} \textbf{Optimality} \\[0.5ex]
\midrule
Kučera (1985)~\cite{Kucera1985} & oblivious &\hspace{0.3cm} No\\[1.5ex]
Gács (1986)~\cite{Gacs1986}& oblivious&\hspace{0.3cm} optimal rate   \\[2ex]
Doty (2008)~\cite{Doty2008} & adaptive&\hspace{0.3cm} optimal rate  \\[2ex]
Barmpalias and Lewis-Pye (2016)~\cite{Barmpalias-Lewis-Pye2018}& oblivious&\hspace{0.3cm} optimal \\[2ex]
Barmpalias, Lewis-Pye, Fang (2016)~\cite{Barmpalias-Lewis-Fang2016}& oblivious& \hspace{0.3cm}optimal for c.e. \\[2ex]
Barmpalias and Downey (2017)~\cite{Barmpalias-Downey2017}& adaptive& \hspace{0.3cm} optimal for left c.e. \\[2ex]
Barmpalias and Lewis-Pye (2019)~\cite{Barmpalias-Lewis-Pye2019}& adaptive& \hspace{0.3cm} ``almost optimal''  \\[2ex]
\bottomrule
\end{tabular}
}\endgroup
\caption{Oracle use for coding by random sequences. The words ``almost optimal'' mean that there are two results that show that the bound is almost tight, with $O(\log n)$ overhead or with computable upper bounds, but the conjecture mentioned above is not proven.}
\label{propofass}
\end{center}
\end{table}

\subsubsection*{Comparison}

We gather the results mentioned above in Tables~\ref{propofas} and \ref{propofass}. We see that (for the case of oblivious coding) the optimal oblivious overhead can be obtained by replacing a local injective coding used in~\cite{Kucera1985,Gacs1986} by some non-local scheme. It seems that non-locality may be necessary even for getting $O(\log n)$ overhead.

For the case of adaptive coding, the arguments from~\cite{Ryabko1986,Doty2008} also give a large overhead ($o(n)$, but not $O(\log n)$, in the case of Doty's argument), that is then combined with additional layer of oblivious coding. This is enough to get the optimal \emph{rate}. Still the arguments from~\cite{Barmpalias-Lewis-Pye2019} give much smaller overhead. The difference between the methods in \cite{Barmpalias-Lewis-Pye2019} and \cite{Doty2008} is striking in the case of logarithmic complexity of prefixes when the overhead $O(\sqrt{n}\log n)$ is large compared to the complexity of prefixes. If, for example, $\KP(\alpha\uhr n)\le 3\log n$ for all $n$, the new construction gives $3\log n$ oracle use, while the previous ones gave more than $\sqrt{n}\log n$.

Let us note also the Levin's construction also gives the result from~\cite{Barmpalias-Lewis-Pye2019} for the special case when $K^*$ is a monotone computable upper bound for $\KP$. (We may always assume without loss of generality that the domain of $K^*$ is a tree.) It would be interesting to bridge the remaining gap by any of these techniques, or to modify Levin's argument to get the best currently known bounds.

\end{document}